\documentclass[twoside, 12pt]{article}
\usepackage{amsthm, amssymb, amscd, mathptmx}
\usepackage[all]{xy}\CompileMatrices\SelectTips{cm}{12}

\theoremstyle{plain}
\newtheorem{Thm}{\sc Theorem}[section]
\newtheorem{Cor}[Thm]{\sc Corollary}
\newtheorem*{Cor*}{\sc Corollary}
\newtheorem{Prop}[Thm]{\sc Proposition}
\newtheorem*{Prop*}{\sc Proposition}
\newtheorem{Lemma}[Thm]{\sc Lemma}

\theoremstyle{definition}

\theoremstyle{remark}
\newtheorem{Remark}[Thm]{Remark}

\newcommand{\cL}{{\cal L}}
\newcommand{\A}{{\cal A}}
\newcommand{\C}{{\cal C}}
\renewcommand{\H}{{\cal H}}
\newcommand{\M}{{\cal M}}
\renewcommand{\O}{{\cal O}}
\renewcommand{\P}{{\cal P}}

\newcommand{\QQ}{\mathbb{Q}}
\newcommand{\Gm}{\mathbb{G}_m}
\newcommand{\codim}{\mathop{\rm codim}}
\newcommand{\ext}{{\mathop{\rm Ext}}}
\newcommand{\Ext}{{\mathop{{\cal E}xt}}}
\newcommand{\GL}{\mathop{\rm GL}}
\newcommand{\SL}{\mathop{\rm SL}}
\newcommand{\mod}{\mathop{\rm mod\,}}
\newcommand{\Hom}{{\mathop{{\cal H}om}}}
\newcommand{\hot}{{\widehat{\otimes}}}
\newcommand{\PGL}{\mathop{\rm PGL}}
\newcommand{\Pic}{{\mathop{\rm Pic}}}
\newcommand{\PIC}{ {\mathop{{\cal P}ic}}}
\newcommand{\Quot}{\mathop{\rm Quot}}
\newcommand{\rk}{\mathop{\rm rk}}
\newcommand{\ch}{\mathop{\rm ch}}
\newcommand{\Td}{\mathop{\rm Td}}
\newcommand{\Sch}{\mathop{\rm Sch}}
\newcommand{\Sets}{\mathop{\rm Sets}}
\newcommand{\Spec}{\mathop{\rm Spec}}
\newcommand{\sm}{{\mathop{\rm sm}}}
\newcommand{\ti}{\tilde}

\newcommand{\Tors}{\mathop{\rm Torsion}}

\begin{document}

\pagestyle{myheadings}

\markboth{\rm A.\ Langer}{\rm  Moduli spaces of principal bundles on singular varieties}

\title{Moduli spaces  of principal bundles on singular varieties}
\author{Adrian Langer}
\date{05.01.2012}
\maketitle

{{\sc Address:}\\
Institute of Mathematics, Warsaw University,
ul.\ Banacha 2, 02-097 Warszawa, Poland}

\bigskip

\begin{center}
To the memory of Professor Masaki Maruyama.
\end{center}

\begin{abstract}
Let $k$ be an algebraically closed field of characteristic zero.
Let $f:X\to S$ be a flat, projective morphism of $k$-schemes of
finite type with integral geometric fibers. We prove existence of
a projective relative moduli space for semistable singular
principal bundles on the fibres of $f$.

This generalizes the result of A. Schmitt who studied the case
when $X$ is a nodal curve.
\end{abstract}

\let\thefootnote\relax

\section{Introduction}

Let $X$ be a smooth projective variety defined over an algebraically
closed field $k$ of characteristic $0$. In \cite{Ma1} and \cite{Ma2} M.
Maruyama, generalizing Gieseker's result from the surface case, constructed 
coarse moduli spaces of semistable sheaves on $X$ (in fact the construction worked  
in some other cases). Later these moduli spaces were also constructed 
for arbitrary varieties (see C. Simpson's paper \cite{Si}) and in an arbitrary characteristic 
(see \cite{La1} and \cite{La2}). Since the moduli space of semistable sheaves compactifies 
the moduli space of (semistable) vector bundles, it is an obvious problem to try to construct 
similar compactifications in case of principal bundles. This problem was considered by many authors
(see \cite{Sch5} and the references within) and it was solved in case of smooth varieties.
However, in case of singular varieties the problem is still open in spite of some partial 
results (see, e.g., \cite{Bh} and \cite{Sch4}). The aim of this paper is to solve this problem 
in the characteristic zero case.

Let $\rho\colon G \to {\GL}(V)$
be a faithful $k$-representation of the reductive group $G$.
In the following we assume that image of the representation $\rho$ is 
contained in $\SL (V)$.

A \emph{pseudo $G$-bundle} is a pair $(\A, \tau)$, where $\A$ is a
torsion free $\O_X$-module of rank $r=\dim V$ and $\tau\colon {\rm Sym}
^*(\A\otimes V)^G \to {\O}_X$ is a nontrivial homomorphism of
${\O}_X$-algebras. In \cite{Bh} U. Bhosle, following earlier work
of A. Schmitt \cite{Sch2} in the smooth case, constructed the
moduli space of pseudo $G$-bundles in case $X$ satisfies some
technical condition, which she showed to hold for seminormal or
$S_2$-varieties. However, it is easy to see that this condition is
always satisfied (see Lemma \ref{Bhosle's-condition}).

Giving the homomorphism $\tau$ is equivalent to giving a section
$$\sigma: X\to {\mathbb H}{\rm om} (\A, V^{\vee}\otimes {\O}_X)/\!/G= \Spec ({\rm Sym} ^* (\A\otimes
V)^G).$$ Let $U_{\A}$ denotes the maximum open subset of $X$ where
$\A$ is locally free. We say that  the pseudo-$G$-bundle $(\A,
\tau)$ is a \emph{singular principal $G$-bundle} if there exists a
non-empty open subset $U\subset U_{\A}$ such that $\sigma (U)
\subset {\mathbb I}{\rm som} (V\otimes {\O}_U,
\A^{\vee}\mid_U)/G$.

In case when $X$ is smooth,  A. Schmitt showed in
\cite{Sch3} that the moduli space of $\delta$-semistable pseudo $G$-bundles 
parametrizes only singular principal $G$-bundles
(for large values of the parameter polynomial $\delta$). 
In a subsequent paper \cite{Sch4}, he also showed that 
in case when $X$ is a curve with only nodes as singularities, 
the moduli space constructed by Bhosle parameterizes only
singular principal $G$-bundles. Moreover, under some mild assumptions
on the representation $\rho$, he proved that $\sigma (U_{\A})
\subset {\mathbb I} {\rm som} (V\otimes {\O}_U,
\A^{\vee}\mid_{U_{\A}})/G$ (in this case we say that $(\A, \tau)$
is an \emph{honest singular principal $G$-bundle}).

In this paper we prove that the same result holds for all the
varieties: the moduli space constructed by Bhosle (for large
values of the parameter polynomial $\delta$) parameterizes 
singular principal $G$-bundles for all
varieties $X$ and all representations $\rho$. More precisely, we
prove the following theorem:

\begin{Thm}
Let $f:X\to S$ be a flat, projective morphism of $k$-schemes of
finite type with integral geometric fibers. Assume that $k$ has
characteristic zero. Let us fix a polynomial $P$ and a faithful
representation $\rho\colon G \to \SL(V)\subset {\GL}(V)$ of the reductive
algebraic group $G$.

\begin{enumerate}
\item There exists a projective moduli space $M^{\rho}_{X/S, P}\to S$ for $S$-flat families
of semistable singular principal $G$-bundles on $X\to S$ such that
for all $s\in S$ the restriction $\A |_{X_s}$ has Hilbert
polynomial $P$.
\item Let $P$ correspond to sheaves of degree $0$.
If the fibres of $f$ are Gorenstein and there exists a
$G$-invariant non-degenerate quadratic form $\varphi$ on $V$ then
$M^{\rho}_{X/S, P}\to S$ parameterizes only honest singular principal
$G$-bundles.
\end{enumerate}
\end{Thm}

Since the fibre of $M^{\rho}_{X/S, P}\to S$ over $s\in S$ is equal
to $M^{\rho}_{X_s, P}$ this theorem shows that moduli spaces of
singular principal bundles are compatible with degeneration.

\medskip

Our approach is similar to the one used in \cite{GLSS1},
\cite{GLSS2} as explained in \cite{Sch5}: we prove a global
boundedness result for swamps (this part of our paper 
works in any characteristic).
Then we use this fact to prove the semistable reduction theorem in the
same way as in the case of smooth varieties. The above mentioned
boundedness result is the main novelty of the paper. It is obtained
by proving that the tensor product of semistable sheaves on a variety 
is not far from being semistable.

The second part of the theorem follows from careful computation of
Hilbert polynomials of dual sheaves on Gorenstein varieties.

\medskip

Unfortunately, the above approach does not work in positive
characteristic because we still do not know how to construct
moduli spaces of swamps for representations of type
$\rho_{a,b,c}\colon {\rm GL} (V) \to {\rm GL}((V^{\otimes a})^{\oplus b}\otimes (\det V)^{-c})$ for $c\ne 0$. In case of
characteristic zero, to construct the moduli space of pseudo
$G$-bundles it was sufficient to use moduli spaces of $\rho
_{a,b,c}$-swamps for $c=0$. But the construction used the Reynolds
operator which is not available in positive characteristic.

Moreover, in positive characteristic there appears a serious
problem with defining the pull-back operation for families of
pseudo $G$-bundles on non-normal varieties (see \cite[Remark
2.9.2.23]{Sch5}).
\medskip

The structure of paper is as follows. In Section 2 we recall some
definitions and results, and we show that Bhosle's condition is
satisfied for all varieties. In Section 3 we study Picard schemes
in the relative setting and we state some existence results for
moduli spaces of swamps. Section 4 is a technical heart of the
paper: we prove that the tensor product of semistable sheaves on
non-normal varieties is close to being semistable. Then in Section
5 we show that  in many cases singular principal bundles of degree
$0$ are honest. In Section 6 we use all these results to prove
semistable reduction theorem and to show existence of projective
relative moduli spaces for (honest) singular principal bundles.

\medskip
\emph{Notation.}

All the schemes in the paper are locally noetherian. A
\emph{variety} is an irreducible and reduced separated scheme of
finite type over an algebraically closed field.

\section{Preliminaries}

\subsection{Basic definitions}

Let $X$ be a $d$-dimensional projective variety over an
algebraically closed field $k$. Let $\O _X(1)$ be an ample line
bundle on $X$.

We say that a coherent sheaf $E$ on $X$ is \emph{torsion free} if
it is pure of dimension $d$. For a torsion free sheaf $E$ we can
write its Hilbert polynomial as
$$P(E)(m):=\chi(X, E\otimes \O_X (m))=\sum_{i=0}^d{\alpha _i(E)}{m^i\over i!}.$$
The \emph{rank} of $E$ is defined as the dimension of $E \otimes
K(X)$, where $K(X)$ is the field of rational functions. It
is denoted by $\rk E$ and it is equal to ${\alpha _d(E)/
\alpha _d(\O_X)}.$ We also define the \emph{degree} of $E$ as
$$\deg E=\alpha_{d-1}(E)-\rk E \cdot \alpha_{d-1}(\O _X)$$
(see \cite[Definition 1.2.11]{HL}). 
The \emph{slope} $\mu (E)$ is, as usually, defined as the
quotient of the degree of $E$  by the rank of $E$.

For two coherent sheaves $E, F$ on $X$ we set $$E\hot F=E\otimes
F/\Tors. $$

\begin{Lemma}\label{normal-tensor-product}
If $X$ is a normal  variety and $E$ and $F$ are torsion free sheaves on $X$ then
$$\mu (E\hot F)=\mu(E)+\mu (F).$$
\end{Lemma}

\begin{proof}
If $E$ is a torsion free sheaf then for a general choice of hyperplanes $H_1,..., H_d\in |\O_X(1)|$
we have $$P(E)(m)= \sum _{i=0}^d\chi (E|_{\bigcap _{j\le i} H_j}){m+i-1 \choose i} $$
(see \cite[Lemma 1.2.1]{HL}). It follows that the rank and degree of $E$ depend only on 
$\chi (E|_{\bigcap _{j\le i} H_j})$ for $i=d$ and $i=d-1$.

If $X$ is a normal variety then by assumption $E$ is locally free outside 
of a closed subset of codimension $\ge 2$. For a general choice of hyperplanes
$H_1,..., H_d\in |\O_X(1)|$ the intersection ${\bigcap _{j\le d} H_j}$ is a union of points and 
${\bigcap _{j\le d-1} H_j}$ is a smooth curve.
Therefore the sheaves $E|_{\bigcap _{j\le i} H_j}$ for $i=d$ and $i=d-1$ are locally free.
Similarly, the sheaves $F|_{\bigcap _{j\le i} H_j}$ for $i=d$ and $i=d-1$ are locally free.
Since in case of points and smooth curves our assertion is clear, we get the lemma.
\end{proof}

\medskip

If $X$ is normal then we can define the determinant of a torsion free sheaf 
$E$ as the reflexivization of $\bigwedge ^{\rk E}E$. In this case
the degree $\deg E$ is equal to the degree of the determinant. This fact follows immediately
from the proof of the above lemma.

\subsection{Serre's conditions $S_k$}

We say that a coherent sheaf $E$ on a scheme $X$ satisfies
\emph{condition $S_k$} if for all points $x\in X$ we have ${\rm
depth}_x(E_x)\ge \min (\dim E_x, k)$.

\medskip

The following lemma is quite standard but we need a more general
version than usual. In case of smooth projective
varieties it is essentially equivalent to \cite[Proposition
1.1.6]{HL}.

\begin{Lemma} \label{Ext-support}
Let $X$ be a Cohen--Macaulay scheme of finite type over a field.
Then \begin{enumerate}
\item
$ \Ext _{X}^q(E,\omega_X)$ is supported on the support of $E$ and
for all points $x\in X$ we have $ \Ext _{X}^q(E,\omega_X)_x=0$ if
$q<\codim _x E$. Moreover, $\codim _x\Ext _{X}^q(E,\omega_X)\ge q$
for $q\ge \codim _x E$.
\item
$E$ satisfies  condition $S_k$ if and only if for all points
$x\in X$ we have $\codim _x\Ext _{X}^q(E,\omega_X)\ge q+k$ for all
$q> \codim _x E$.
\end{enumerate}
\end{Lemma}

\begin{proof}
By assumption $X$ is Cohen--Macaulay and every local ring
$\O_{X,x}$ is a quotient of a regular local ring, so we can apply
the local duality theorem (see \cite[Theorem 6.7]{Ha2}) to prove
that $ \Ext _{X}^q(E,\omega_X)_x\ne 0$ if and only if $\H ^{\dim
_x X-q}_x(E)\ne 0$. But the local cohomology $\H ^{\dim _x
X-q}_x(E)$ vanishes if $\dim _x X-q> \dim _xE$, which proves the
first part of 1. If $q= \codim _x E$ then $\codim _x(\Ext
_{X}^q(E,\omega_X))\ge q$ is equivalent to the obvious inequality
$\dim _x(\Ext _{X}^q(E,\omega_X))\le \dim _xE$. Hence, since every
sheaf satisfies $S_0$, the second part of 1 follows from 2.

\medskip
To prove 2 note that by  \cite[Theorem 3.8]{Ha2}  ${\rm
depth}_x(E_x)\ge \min (\dim E_x, k)$ if and only if $\H ^i_x(E)=0$
for all $i<\min (\dim E_x, k)$. By the local duality theorem this
last condition is equivalent to $ \Ext _{X}^q(E,\omega_X)_x=0$ for
$q>\max (\codim _x E, \dim \O_{X,x}- k)$. This is equivalent to
saying that for $q>\codim _xE$ a non-vanishing of $ \Ext
_{X}^q(E,\omega_X)_x$ implies $\dim \O_{X,x}\ge q+k$.
\end{proof}

\medskip

Let  $k$ be an algebraically closed field. Let $X$ be a
$d$-dimensional pure (i.e., $\O_X$ satisfies $S_1$) 
scheme of finite type over $k$. Let $C$ be a
smooth curve defined over $k$ and let us fix a closed point $0\in C$. 
By $p_X:Z=X\times C\to X$ we denote the projection.
Let $Y$ be a non-empty proper closed
subscheme of $X\times \{0\}$ (in particular, we assume that $X$
has dimension $\ge 1$), and let $i:Y\hookrightarrow Z$ denote the
corresponding closed embedding. Let us also set $U=Z-Y$ and let
$j: U\hookrightarrow Z$ denote the corresponding open embedding.

\begin{Lemma} \label{Bhosle's-condition}
If $E$ is a pure sheaf of
dimension $d$ on $X$ then we have a canonical isomorphism
$ p_X^*E \simeq j_*j^*(p_X^*E)$. In particular, $\O_Z\simeq j_*\O_U$
and for any locally free sheaf $F$ on $Z$ we have $F \simeq j_*j^*F$.
\end{Lemma}

\begin{proof}
Let us set $F=p_X^*E$. Since we have a canonical map $F\to
j_*j^*F$, the assertion is local and hence we can assume that $X$
and $Y$ are affine. By \cite[Proposition 2.2]{Ha2} we have an
exact sequence
$$0\to i_*\H ^0_Y (F) \to F \to j_*j^*F \to i_*\H ^1_Y (F)\to 0.$$
To prove that $i_*\H ^i_Y (F)=0$ for $i=0,1$, it is sufficient to
prove that for every point $y\in Y$, the depth of  $F_{y}$ is at
least $2$ (see \cite[Theorem 3.8]{Ha2}). Now, let us take a local
parameter $s\in \O_{C, 0}$. Then $F_{ y}/s F _{y}\simeq E_{y}$ has
depth at least $1$ (because by assumption $E$ satisfies $S_1$), so
the required assertion is clear.
\end{proof}

\begin{Remark}
The above lemma shows in particular that every variety satisfies
condition (2.19) in the sense of Bhosle (see \cite[Definition
2.8]{Bh}).
\end{Remark}

\subsection{Moduli spaces of pseudo $G$-bundles}

Let us fix a faithful representation $\rho\colon G \to \SL (V)\subset {\rm
GL}(V),\ r= \dim V,$ of a reductive algebraic group $G$.

A \emph{pseudo $G$-bundle} is a pair $(\A, \tau)$, where $\A$ is a
torsion free $\O_X$-module of rank $r$ and $\tau\colon {\rm Sym}
^*(\A\otimes V)^G \to {\O}_X$ is a nontrivial homomorphism of
${\O}_X$-algebras. Giving $\tau$ is equivalent to giving a section
$$\sigma: X\to {\mathbb H}{\rm om} (\A, V^{\vee}\otimes {\O}_X)/\!/G= \Spec ({\rm Sym} ^* (\A\otimes
V)^G).$$

A \emph{weighted filtration} $(\A_{\bullet}, \alpha _{\bullet})$
of $\A$ is a pair consisting of a filtration
$$\A_{\bullet}= (0\subset \A_1\subset \dots \subset \A_s\subset \A)$$
by saturated subsheaves (i.e., such that the quotients $\A/\A_i$
are torsion free) of increasing ranks and an $s$-tuple
$$\alpha _{\bullet} =(\alpha _1,\dots , \alpha _s)$$
of positive rational numbers. To every weighted filtration
$(\A_{\bullet}, \alpha _{\bullet})$ one can associate the
polynomial
$$M(\A_{\bullet}, \alpha _{\bullet}):=\sum _{i=1}^s\alpha _i (P({\A})\cdot \rk (\A_i)- P({\A_i})\cdot \rk (\A)).$$

If $(\A_{\bullet}, \alpha _{\bullet})$ is a weighted filtration of
a pseudo $G$-bundle $(\A, \tau)$ then one can also define the
number $\mu (\A_{\bullet}, \alpha _{\bullet}, \tau)$ describing
stability of the ${\rm SL} (\A\otimes K(X))$-group action on ${\rm
Hom} (\A \otimes K(X), V^{\vee}\otimes K(X))/\!/G$ (see, e.g.,
\cite[3.3.2]{Sch4/5}).

Let us fix a positive polynomial $\delta$ with rational
coefficients and of degree $\le \dim X-1$. Then we say that a
pseudo $G$-bundle $(\A, \tau)$ is \emph{$\delta$-(semi)stable} if
$\A$ is torsion free and for any weighted filtration
$(\A_{\bullet}, \alpha _{\bullet})$ of $\A$ we have inequality
$$M(\A_{\bullet}, \alpha _{\bullet})+\delta \cdot \mu (\A_{\bullet}, \alpha _{\bullet}, \tau) (\ge) 0.$$

To define the slope version of (semi)stability instead of
$M(\A_{\bullet}, \alpha _{\bullet})$ one uses the rational number
$$L(\A_{\bullet}, \alpha _{\bullet}):=\sum _{i=1}^s\alpha _i (\deg {\A}\cdot \rk (\A_i)- \deg {\A_i}\cdot \rk (\A)).$$

The following theorem follows from the results of Schmitt
\cite{Sch2} (in the smooth case) and from the results of Bhosle
\cite{Bh} and Lemma \ref{Bhosle's-condition} in general:

\begin{Thm}\label{Bhosle-Schmitt}
Let $(X, \O _X(1))$ be a polarized projective variety defined over
an algebraically closed field of characteristic zero. Then there
exists a projective moduli space $M ^{\rho, \delta}_{X, P}$ for
$\delta$-semistable pseudo $G$-bundles $(\A, \tau)$ on $X$, such
that $\A$ has Hilbert polynomial $P$ (with respect to $\O_X(1)$).
\end{Thm}

\subsection{Semistability of singular principal $G$-bundles}

Let $(\A ,\tau)$ be a pseudo $G$-bundle. Let us recall that giving
$\tau$ is equivalent to giving a section
$$\sigma: X\to {\mathbb H}{\rm om} (\A, V^{\vee}\otimes {\O}_X)/\!/G= \Spec ({\rm Sym} ^* (\A\otimes
V)^G).$$ Let $U_{\A}$ denotes the maximum open subset of $X$ where
$\A$ is locally free. The pseudo-$G$-bundle $(\A, \tau)$ is a
\emph{singular principal $G$-bundle} if there exists a non-empty
open subset $U\subset U_{\A}$ such that
$$\sigma (U) \subset {\mathbb I}{\rm som} (V\otimes {\O}_U,
\A^{\vee}\mid_U)/G .$$ If $\A$ has degree $0$ and $\sigma (U_{\A})
\subset {\mathbb I}{\rm som} (V\otimes {\O}_{U_{\A}},
\A^{\vee}\mid_{U_{\A}})/G $ then we say that $(\A, \tau)$ is an
\emph{honest singular principal $G$-bundle}.

Let us recall  that a singular principal $G$-bundle $(\A, \tau)$,
via the following pull-back diagram, defines a principal
$G$-bundle $\P (\A , \tau)$ over the open subset $U$:
$$
\xymatrix{ {\cal P}({\cal A},\tau) \ar[r] \ar[d] & {\mathbb I}{\rm
som} (V\otimes{\O}_U,{\A}^\vee \mid_{U}) \ar[d]
\\
U \ar[r]^-{\sigma_{|U}} & {\mathbb I}{\rm
som}(V\otimes{\O}_U,{\A}^\vee \mid_U)/G. }
$$

If $X$ is smooth then every singular principal $G$-bundle is
honest (see \cite[Lemma 3.4.2]{Sch4/5}). Note that our definitions
are slightly different to those appearing in previous literature
(which changed in time to the one close to our definitions).

\medskip

Let $(\A ,\tau)$ be a singular principal $G$-bundle and let
$\lambda: \Gm \to G$ be a one-parameter subgroup of $G$. Let
$$Q_{G}(\lambda):=\{ g\in G: \lim _{t\to \infty}\lambda (t)g\lambda(t)^{-1}\hbox{ exists in }G\}.$$
A \emph{reduction} of $(\A ,\tau)$ to $\lambda$ is a section
$\beta: U' \to \P (\A, \tau)/Q_{G}(\lambda)$ defined over some
non-empty open subset $U'\subset U$. Such reduction defines a
reduction of structure group of a principal $\GL (V)$-bundle
associated to $\A \mid_{U'}$ to the parabolic subgroup $Q _{\GL
(V)}(\lambda)$, so we get a weighted filtration $(\A_{\bullet}',
\alpha _{\bullet})$ of $\A\mid_{U'}$.

Let $j:U'\hookrightarrow X$ denote the open embedding. Then for
$i=1,...,s$ we define $\A_i$ as saturation of $\A\cap j_*(\A_i')$.
In particular, we get a weighted filtration $(\A_{\bullet}, \alpha
_{\bullet})$ of $\A$.

We say that a singular principal $G$-bundle $(\A, \tau)$ is
\emph{(semi)stable} if $\A$ is torsion free and for any reduction
of $(\A ,\tau)$ to a one-parameter subgroup $\lambda:\Gm \to G$ we
have inequality
$$M(\A_{\bullet}, \alpha _{\bullet}) (\ge) 0.$$

\section{Moduli spaces of swamps revisited}

In this section we recall and reprove some basic results
concerning existence of the relative Picard scheme and its
compactifications. Then we apply these results to existence of
moduli spaces of swamps.

We interpret the compactified Picard scheme as the coarse moduli
space of stable rank $1$ sheaves and we use Simpson's construction
of these moduli spaces to prove existence of the universal family
(i.e., the Poincare sheaf) under appropriate assumptions. This
approach, although very natural, seems to be hard to find in
existing literature, especially in the relative case.

\medskip

The notation in this section is as follows. $R$ denotes a
universally Japanese ring. We also fix  a projective morphism $f:
X\to S$ of $R$-schemes of finite type with geometrically connected
fibers. We assume that $f$ is of pure relative dimension $d$. By
$\O_X(1)$ we denote an $f$-very ample line bundle on $X$. We also
fix a polynomial $P$.

\subsection{Universal families on relative moduli spaces}

Let us define the moduli functor $\M _{X/S,P}: (\Sch
/S)\longrightarrow (\Sets)$ by sending $T\to S$ to
$$\M _{X/S,P}(T)=\left\{
\begin{array}{l}
\hbox{isomorphism classes of $T$-flat families of Gieseker}\\
\hbox{semistable sheaves with Hilbert polynomial $P$}\\
\hbox {on the geometric fibres of }p: T\times_S X\to T \\
\end{array}
\right\} /\sim ,$$ where $\sim$ is the equivalence relation $\sim$
defined by $F\sim F'$ if and only if there exists an invertible
sheaf $K$ on $T$ such that $F\simeq F'\otimes p^*K$.

\begin{Thm} \emph{(see \cite{Ma1}, \cite{Ma2}, \cite{Si}, \cite{La1} and \cite{La2})}
There exists a projective $S$-scheme $M_{X/S,P}$, which uniformly
corepresents the functor $\M_{X/S,P}$.  Moreover, there is an open
subscheme $M^{s}_{X/S,P}\subset M_{X/S,P}$ that universally
corepresents the subfunctor $\M^{s}_{X/S,P}$ of families of
geometrically Gieseker stable sheaves.
\end{Thm}

We are interested when the moduli scheme $M^{s}_{X/S,P}$
represents the functor $\M^{s}_{X/S,P}$. This is equivalent to
existence of a universal family on $M^{s}_{X/S,P}\times_S X$.

Let us recall that the moduli scheme $M^{s}_{X/S,P}$ is
constructed as a quotient of an appropriate subscheme $R^s$ of the
Quot-scheme $\Quot (\H; P) $ by $\PGL (V).$ Let $q^*\H\to {\ti F}$
denote the universal quotient on $R^s\times _SX$.

\begin{Prop}\label{universal_family}
{\rm (\cite[Proposition 4.6.2]{HL})} The moduli scheme
$M^{s}_{X/S,P}$ represents the functor $\M^{s}_{X/S,P}$ if and
only if there exists a $\GL (V)$-linearized line bundle $A$ on
$R^s$ on which elements $t$ of the centre $Z(\GL (V))\simeq \Gm$
act via multiplication by $t$. If such $A$ exists then $\Hom
(p^*A,{\ti F})$ descends to a universal family and any universal
family is obtained in such a way.
\end{Prop}

\subsection{Existence of compactified Picard schemes in the relative case}

For simplicity we assume that all geometric fibers of $f$ are
irreducible and reduced (hence they are varieties) and that $S$ is
connected.

Let us fix a polynomial $P$. For all locally noetherian $S$-schemes
$T\to S$ let us set
$$\PIC'_{X/S,P}(T)=\left\{
\begin{array}{l}
\hbox{isomorphism classes of invertible sheaves $L$ on $X_T=T\times _S X$}\\
\hbox{such that $\chi (X_t, L_t (n))=P(n)$ for every geometric $ t\in T$}\\
\end{array}
 \right\}.$$ Note that if $\PIC'_{X/S,P}(T)$ is non-empty then the
highest coefficient of $P$ is the same as the highest coefficient of
the Hilbert polynomial of $\O_{X_s}$ for any $s\in S$.

As before we introduce an equivalence relation $\sim$ on $\PIC'_{X/S,P}(T)$ by
$L\sim L'$ if and only if there exists an invertible sheaf $K$ on $T$
such that $L\simeq L'\otimes p^*K$. Then we can define \emph{the Picard functor}
$$\PIC_{X/S, P}: (\Sch /S)\longrightarrow (\Sets)$$
by sending an $S$-scheme $T$ to $\PIC_{X/S, P}(T)=\PIC'_{X/S, P}(T)/\sim$

\medskip

Let us also define the compactified relative Picard functors. There
are two different methods of compactification of the Picard scheme.
We can compactify the Picard scheme by adding all the rank $1$ torsion
free sheaves on the fibres of $X$ or only those rank $1$ torsion free
sheaves that are locally free on the smooth locus of the fibres. The
second method has the advantage of producing a smaller scheme.

Let us set
$${\overline{\PIC '}}_{X/S,P}(T)=\left\{
\begin{array}{l}
\hbox{isomorphism classes of $T$-flat sheaves $L$ on $X_T=T\times _S X$}\\
\hbox{such that $L_t$ is a torsion free, rank $1$ sheaf on $X_t$}\\
\hbox{and $\chi (X_t, L_t (n))=P(n)$ for every geometric $ t\in T$}\\
\end{array}
 \right\}.$$
As before we define  \emph{the compactified Picard functor}
$${\overline{\PIC}}_{X/S, P}: (\Sch /S)\longrightarrow (\Sets)$$ by
sending an $S$-scheme $T$ to $\overline{\PIC}_{X/S,
P}(T)=\overline{\PIC'}_{X/S, P}(T)/\sim$.

We also define  \emph{the small compactified Picard functor}
$${\overline{\PIC}}_{X/S, P}^{\sm}: (\Sch /S)\longrightarrow (\Sets)$$ by
sending an $S$-scheme $T$ to $$\overline{\PIC}_{X/S,
P}^{\sm}(T)=\left\{
\begin{array}{l}
L\in \overline{\PIC'}_{X/S, P}(T) \hbox{ such that
$L$ is locally free }\\
\hbox{on the smooth locus of $X_T/T$}\\
\end{array}
\right\}/\sim .$$

\begin{Thm} \label{Pic}
Assume that $f:X\to S$ has a section $g: S\to X$.

\begin{enumerate}
\item
There exists a quasi-projective $S$-scheme $\Pic_{X/S, P}$ that represents
the Picard functor $\PIC_{X/S, P}$.
\item
If $g (S)$ is contained in the smooth locus of $X/S$ then there
exists a projective $S$-scheme $\overline{\Pic}_{X/S, P}$ that
represents the compactified Picard functor $\overline{\PIC}_{X/S,
P}$. Moreover, $\overline{\Pic}_{X/S, P}$ contains a closed
$S$-subscheme $\overline{\Pic}^{\sm}_{X/S, P}$ that represents the
small compactified Picard functor $\overline{\PIC}^{\sm}_{X/S,
P}$.
\end{enumerate}
\end{Thm}

\begin{proof}
First let us remark that all the Picard functors ${\PIC}_{X/S, P}$,
$\overline{\PIC}_{X/S, P}$ and $\overline{\PIC}^{sm}_{X/S, P}$ are
subfunctors of the moduli functor $\M_{X/S, P}$. In fact, from our
assumptions it follows that $\overline{\PIC}_{X/S, P}=\M^s_{X/S,
P}=\M_{X/S, P}$.  Now we can construct ${\Pic}_{X/S, P}$,
$\overline{\Pic}_{X/S, P}$ and $\overline{\Pic}^{sm}_{X/S, P}$ as
Geometric Invariant Theory quotients of appropriate subschemes
$R_{\Pic}\subset R_{{\overline{\Pic}^{sm}}}\subset R_{\overline{\Pic }
}=R^s=R^{ss}$ of the Quot-scheme used to construct the moduli space
$M^s_{X/S, P}$ by $\GL (V)$. In fact all these quotients are $\PGL
(V)$-principal bundles.  To prove that $\overline{\Pic}^{\sm}_{X/S,
P}$ is a closed subscheme of $\overline{\Pic}_{X/S, P}$ it is
sufficient to see that $R_{{\overline{\Pic}^{sm}}}$ is a closed
subscheme of $ R_{\overline{\Pic } }$. This follows from \cite[Lemma
on p.~37]{AK2} applied to the universal quotient restricted to the
smooth locus of $R_{\overline{\Pic }}\times _SX\to R_{\overline{\Pic
}}$.

To prove 1 by (a slight generalization of) Proposition
\ref{universal_family} it is sufficient to show existence of a $\GL
(V)$-linearized line bundle $A_{\Pic}$ on $R_{\Pic}$ on which the
centre of $\GL (V)$ acts with weight $1$.

Let us set $A_{\Pic}=\det p_{*}({\ti F}\otimes q^*\O_{g (S)})$,
where $\ti F$ comes from the universal quotient on
$R_{\Pic}\times_S X$. The definition makes sense since $\ti F$ is
a line bundle on $R_{\Pic}\times_S X$ and $p_{*}({\ti F}\otimes
q^*\O_{g (S)})= (id_{R_{\Pic}}\times _S g)^*{\ti F}$ is also a
line bundle. The centre of $\GL (V)$ acts on the fibre of
$A_{\Pic}$ at $([\rho], x)\in R_{\Pic}\times _S X$ with weight
$\chi (\O_{X_{f(x)}}|_x)=1$, which implies the first assertion of
the theorem.

Now assume that $g (S)$ is contained in the smooth locus of $X/S$.
Then the same argument as above gives existence of the Poincare
sheaf on $\overline{\Pic}^{\sm}_{X/S, P}$. Existence of the
Poincare sheaf on $\overline{\Pic}_{X/S, P}$ is slightly more
difficult. First let us show that there exists a resolution
$$0\to E_n\to \dots \to E_0\to \O_{g (S)}\to 0,$$ where $E_i$ are
locally free sheaves on $X$. Since there are sufficiently many
locally free sheaves on $X$ we can construct the resolution up to
step $E_{n-1}$, where $n$ is the relative dimension of $X/S$. Then
the kernel of $E_{n-1}\to E_{n-2}$ is also locally free. Indeed,
it is sufficient to check it on the geometric fiber $X_s$ over
$s\in S$, where one can use the fact that the homological
dimension of $\O_{g(s)}$ is equal to $n$ (this follows from
the smoothness assumption).

Tensoring with a high tensor power $\O_X(m)$ we can assume that all the
higher direct images of ${\ti F}\otimes q^*(E_i(m))$ under the
projection $p$ vanish. In particular, all sheaves $p_*({\ti F}\otimes
q^*(E_i (m)))$ are locally free. Then we can set
$$A_{\overline{\Pic}}=\det p_{!}({\ti F}\otimes q^*(\O_{g (S)}(m)))=
\bigotimes _i(\det p_*({\ti F}\otimes q^*(E_i (m))))^{(-1)^i}.$$ Obviously,
the centre of $\GL (V)$ still acts on the fibres of $A_{\overline{\Pic}}$
with weight $1$. Hence the theorem follows from Proposition \ref{universal_family}.
\end{proof}

\medskip

\begin{Remark}
Note that the second part of Theorem \ref{Pic} does not
immediately follow from \cite{AK1} and \cite{AK2}.
Representability of (compactified) Picard functors is proven there
only in \'etale topology or after rigidification (see, e.g.,
\cite[Theorems 3.2 and 3.4]{AK2}). Rigidification of the
compactified Picard functor amounts in our case to restricting to
the open subset of $R_{\overline{\Pic}}$, where the restriction of
$\ti F$ to $g (S)$ is invertible. Then by the same argument as in
the proof of 1 of Theorem \ref{Pic} we can construct the scheme
representing the corresponding rigidified Picard functor obtaining
\cite[Theorem 3,4]{AK2}. However, we prefer to make a stronger
assumption as in 2 to construct the projective Picard scheme.
\end{Remark}

\subsection{Moduli spaces of swamps}

Let us fix non-negative integers $a$ and $b$ and consider a $\GL
(V)$-module $(V^{\otimes a})^{\oplus b}$. Let $\rho_{a,b}\colon
{\rm GL} (V) \to {\rm GL}(V^{\otimes a})^{\oplus b})$ be the
corresponding representation. If $\A$ is a sheaf of rank $r=\dim
V$ then we can associate to it a sheaf $ \A_{\rho
_{a,b}}=(\A^{\otimes a})^{\oplus b}$. On the open set where $\A$
is locally free, $\A_{\rho _{a,b}}$ is a locally free sheaf
associated to the principal bundle obtained by extension from the
frame bundle of $\A$.

Let us recall that a \emph{$\rho_{a,b}$-swamp} is a triple $(\A,
L, \varphi)$ consisting of a torsion free sheaf $\A$ on $X$, a
rank $1$ torsion free sheaf $L$ on $X$ and a non-zero homomorphism
$\varphi : \A_{\rho _{a,b}}\to L$.

Let us fix a positive polynomial $\delta$ of degree $\le d-1$ with
rational coefficients. Let us write
$\delta(m)={\overline{\delta}}\frac{m^{d-1}}{(d-1)!}+O(m^{d-2})$.

For a weighted filtration $(\A_{\bullet}, \alpha _{\bullet})$ of
$\A$ we set $r_i=\rk \A_i$ and we consider a vector $\gamma \in
\QQ ^r$ defined by
$$\gamma =\sum \alpha _i
\bigl(\underbrace{r_i-r,...,r_i-r}_{r_i\times},\underbrace{r_i,...,r_i}_{(r-r_i)\times}\bigr)
.
$$
Let $\gamma _j$ denote the $j$th component of $\gamma$. We set
$$
\mu\bigl(\A_\bullet,{\alpha}_\bullet;\varphi\bigr)=-\min\Bigl\{\,\gamma_{i_1}+\cdots+\gamma_{i_{a}}\,\big|\,
(i_1,...,i_{a})\in I: \varphi_{|(\A_{i_1}\otimes \cdots\otimes
\A_{i_{a}})^{\oplus b}} \not\equiv 0\,\Bigr\},$$ where $I=\{1,...,
s+1 \}^{\times a}$ is the set of all multi-indices.

Let us recall that a $\rho_{a,b}$-swamp $(\A, L, \varphi)$ is
\emph{$\delta$-(semi)stable} if for all weighted filtrations
$(\A_{\bullet}, \alpha _{\bullet})$ we have
$$M(\A_\bullet,{\alpha}_\bullet)+\mu\bigl(\A_\bullet,{\alpha}_\bullet;\varphi\bigr) \delta (\ge ) 0 .$$

A $\rho_{a,b}$-swamp $(\A, L, \varphi)$ is \emph{slope
$\overline{\delta}$-(semi)stable} if for all weighted filtrations
$(\A_{\bullet}, \alpha _{\bullet})$ we have
$$L(\A_\bullet,{\alpha}_\bullet)+\mu\bigl(\A_\bullet,{\alpha}_\bullet;\varphi\bigr) \overline{\delta} (\ge ) 0 .$$

\medskip

Now we can state the most general existence result for moduli
spaces of swamps. We keep the notation from the beginning of this
section.

\begin{Thm}
Let us fix an $S$-flat family $\cL$ of pure sheaves of dimension
$d$ on the fibres of $f: X\to S$. Assume that either $d=1$ or $f$
has only irreducible and reduced geometric fibres. Then there
exists a coarse $S$-projective moduli space for
$\delta$-semistable $S$-flat families of $\rho_{a,b}$-swamps $(\A,
\cL, \varphi)$ such that for every $s\in S$ the restriction
$\A|_{X_s}$ has Hilbert polynomial $P$.
\end{Thm}

In case when $X$ is a smooth complex projective variety this
theorem was proved by G\'omez and Sols in \cite{GS}, and later
generalized by Bhosle to singular complex varieties satisfying
Bhosle's condition in \cite{Bh}. Note that in \cite{GS} and
\cite{Bh} the authors considered only the case when $\cL$ is
locally free. However, this is not necessary due to Lemma
\ref{Bhosle's-condition} and it is sufficient to assume that $\cL$
is torsion free. Generalization to the relative case in arbitrary
characteristic follows from \cite{La1} and \cite{La2}. We need
only to comment why one does need to require that the fibres of
$f$ are irreducible or reduced in the curve case. This fact
follows from \cite[Remark 4.4.9]{HL}: torsion submodules for
sheaves on curves are detected by any twist of its global
sections. This allows to omit using \cite[Proposition 2.12]{Bh} in
the curve case. In particular, this shows that all the results of
Sorger \cite{So} are now a part of the more general theory.

\medskip

We also have another variant of the above theorem (cf.
\cite[Theorem 2.3.2.5]{Sch5}):

\begin{Thm}
Let us fix a Hilbert polynomial $Q$. Assume that all geometric
fibers of $f$ are irreducible and reduced and assume that $f:X\to
S$ has a section $g: S\to X$ such that $g (S)$ is contained in the
smooth locus of $X/S$. Then there exists a coarse
moduli space for $\delta$-semistable $S$-flat families of
$\rho_{a,b}$-swamps $(\A, \cL, \varphi)$ such that for every $s\in
S$ the restriction $\A|_{X_s}$ has Hilbert polynomial $P$ and the
restriction $\cL |_{X_s}$ has Hilbert polynomial $Q$. This moduli space
is projective over  $\overline{\Pic}_{X/S, Q}$.
\end{Thm}

\section{Tensor product of semistable sheaves on non-normal varieties}

Let $(X, \O_X (1))$ be a $d$-dimensional polarized projective
variety defined over an algebraically closed field $k$.

Let $\nu :{\ti X}\to X$ denote the normalization of $X$ and let
$E$ be a coherent $\O_X$-module. Since $\nu$ is a finite morphism,
there exists a well defined coherent $\O_{\ti X}$-module $\nu
^{!}E$ corresponding to the $\nu_*{\O_{\ti X}}$-module $\Hom (\nu
_*{\O_{\ti X}}, E)$.  If $E$ is torsion free then we have $\Hom
_{\O_X} (\nu _*{\O_{\ti X}}/\O_X, E)=0$. Hence
$$\nu_* (\nu ^{!}E)=\Hom _{\O_X} (\nu _*{\O_{\ti X}}, E)\subset \Hom
_{\O_X} ({\O_{X}}, E)=E$$ and $\nu ^{!}E$ is also torsion free.

\begin{Lemma}\label{mu-mu}
There exists a constant $\alpha$ (depending only on the variety
$X$) such that for any rank $r$ torsion free sheaf $E$ on $X$ we
have
$$0\le \mu (E) -\mu (\Hom (\nu _*\O_{\ti X}, E))\le \alpha.$$
\end{Lemma}

\begin{proof}
We have an exact sequence
$$0\to \Hom _{\O_X} (\nu _*{\O_{\ti X}}, E)\to E\to \Ext
^1_{\O_X}({\nu _*{\O_{\ti X}/\O_X}},E).$$  For
large $m$ we have
$$P(\Hom _{\O_X} (\nu _*{\O_{\ti X}}, E))(m)\le P(E)(m)$$
and, since $\Hom _{\O_X} (\nu _*{\O_{\ti X}}, E)$ and $E$ have the same rank,
we have
$$\mu (\Hom _{\O_X} (\nu _*{\O_{\ti X}}, E))\le \mu (E).$$
On the other hand we have
$$\alpha _{d-1}(E)\le \alpha _{d-1}(\Hom _{\O_X} (\nu _*{\O_{\ti X}},
E))+\alpha _{d-1}(\Ext ^1_{\O_X}({\nu _*{\O_{\ti X}/\O_X}},E)).$$
Note that $\Ext ^1_{\O_X}({\nu _*{\O_{\ti X}/\O_X}},E)$ is
supported on the support of ${\nu _*{\O_{\ti X}/\O_X}}$. Let
$Y_1,\dots ,Y_k$ denote codimension $1$ irreducible components of
the support of ${\nu _*{\O_{\ti X}/\O_X}}$.  Then $\alpha
_{d-1}(\Ext ^1_{\O_X}({\nu _*{\O_{\ti X}/\O_X}},E))$ can be
bounded from the above using the ranks of $\Ext ^1_{\O_X}({\nu
_*{\O_{\ti X}/\O_X}},E)$ at $Y_1, \dots, Y_k$. Hence by the above
inequality, to prove the lemma it is sufficient to bound these
ranks.

There exists a subsheaf $G\subset E$ such that $G$ is locally free
(we need only locally free in codimension $1$) and $E/G$ is
torsion (i.e., equal to zero at the generic point of $X$).  This
can be constructed by taking $r$ general sections of $E(m)$ for
large $m$ and twisting the image of $\O_X^r\subset
H^0(E(m))\otimes \O_X\to E(m)$ by $\O_X (-m)$.

Then we have an exact sequence
$$0=\Hom (\nu _* \O_{\ti X}/\O_X, E)\to \Hom (\nu _* \O_{\ti X}/\O_X,
E/G)\to \Ext^1 (\nu _* \O_{\ti X}/\O_X, G)$$ Note that the sheaves
in this sequence are supported on $\bigcup Y_i$ and the rank of $
\Ext^1 (\nu _* \O_{\ti X}/\O_X, G)$ on $Y_i$ is the same as the
rank of $ \Ext^1 (\nu _* \O_{\ti X}/\O_X, \O_X^r)$ on $Y_i$. In
particular, it depends only on the rank $r$ and it is independent
of $E$. Hence the dimensions of $\Hom (\nu _* \O_{\ti X}/\O_X,
E/G)$ at the generic points of $Y_1, \dots , Y_k$ are bounded from
the above by a linear function of $r$. But this implies that the
ranks of $E/G$, and hence also of $\Ext^1 (\nu _* \O_{\ti X}/\O_X,
E/G)$, on $Y_1, \dots , Y_k$ are bounded independently of $E$. Now
we can use the sequence
$$ \Ext^1 (\nu _* \O_{\ti X}/\O_X, G)\to \Ext^1 (\nu _* \O_{\ti
X}/\O_X, E)\to \Ext^1 (\nu _* \O_{\ti X}/\O_X, E/G)$$ to bound the
ranks of $\Ext ^1_{\O_X}({\nu _*{\O_{\ti X}/\O_X}},E)$ on $Y_1, \dots,
Y_k$.
\end{proof}

\begin{Cor}\label{useful}
Let us set $\beta=\alpha_{d-1} (\O_{\ti X})-\alpha _{d-1}(\O_X)$.
 Then for any rank
$r$ torsion free sheaf $E$ on $X$ we have
$$\beta\le \mu (E)-\mu (\nu ^{!}E) \le \alpha+\beta,$$ where the slopes are computed
with respect to $\O_X(1)$ on $X$ and $\nu^*\O_X(1)$ on $\ti X$.
\end{Cor}

\begin{proof}
For any sheaf $F$ on $\ti X$ we have
$$\chi ({\ti X}, F\otimes \nu^*\O_X(m))= \chi (X, \nu_*F\otimes
\O_X(m)).$$ This implies that
$$\mu (\nu_*F)-\mu (F)=\alpha_{d-1} (\O_{\ti X})-\alpha _{d-1}(\O_X)=\beta.$$
Therefore, since
$$\nu_* (\nu ^{!}E)=\Hom _{\O_X} (\nu _*{\O_{\ti X}}, E),$$
we have
$$\begin{array}{rcl} \mu (E)-\mu (\nu ^{!}E)&=& (\mu (E)-\mu
(\Hom (\nu _*\O_{\ti X}, E)))+(\mu(\nu_* (\nu
^{!}E))-\mu (\nu ^{!}E))\\
&=& (\mu (E)-\mu (\Hom (\nu _*\O_{\ti X}, E)))+\beta.
\end{array}$$
Now the corollary follows from Lemma \ref{mu-mu}.
\end{proof}

\begin{Cor} \label{useful2}
For any rank $r$ torsion free sheaf $E$ on $X$ we have
$$\beta \le \mu _{\max}(E)-\mu _{\max}(\nu ^{!}E)\le \alpha+
\beta .$$
\end{Cor}

\begin{proof}
If $G\subset E$ is a subsheaf of $E$ then $\nu^{!}G\subset
\nu^{!}E$ and hence
$$\mu (G)\le \mu (\nu^{!}G)+\alpha+
\beta\le \mu _{\max}(\nu^{!}E)+\alpha+ \beta.$$ This proves that
$$\mu _{\max}(E)\le  \mu _{\max}(\nu^{!}E)+\alpha+
\beta.$$ Now if $F\subset \nu^{!}E$ then $\nu_*F\subset
\nu_*(\nu^{!}E)\subset E$. Therefore
$$\mu (F)=\mu (\nu_*F)-\beta\le \mu _{\max}(E)-\beta,$$ which implies that
$$\mu _{\max}(\nu ^{!}E)\le \mu _{\max}(E)-\beta .$$
\end{proof}

\medskip
 For a torsion free sheaf $E$ on $X$ we set $\nu
^{\sharp} E=\nu ^*E/\Tors$. Then $\nu_*{\nu^{\sharp} E}=
(\nu_*{\nu^*E})/\Tors$.

Note that $\nu^{!}$ is an equivalence of categories of sheaves on
$X$ and $\ti X$ whereas $\nu ^{\sharp}$ has much worse properties.
But  $\nu ^{\sharp}$ has the following important property: since
$\nu^*(E_1\otimes E_2)=\nu^*E_1\otimes \nu^*E_2$ we have
$\nu^{\sharp}(E_1\hot E_2)=\nu^{\sharp}E_1\hot \nu^{\sharp}E_2$.

\medskip

Let $\C= {\rm Ann} (\nu_* \O_{\ti X}/\O_X)\subset \O_X$ and $\C
_{\ti X}=\C \cdot \O_{\ti X}\subset \O_{\ti X}$ denote conductor
ideals of the normalisation.

\begin{Lemma}\label{mu-mu2}
For any torsion free sheaf $E$ on $X$ we have
$$  \mu (\nu^{\sharp} E) \le \mu (\nu^!{ E})- \mu (\C_{\ti X}).$$
\end{Lemma}

\begin{proof}
Note that $\C =\Hom _{\O_X} (\nu _*{\O_{\ti X}}, \O _X)$.
Therefore for any coherent $\O_X$-module $E$ we have a canonical
map
$$\C \otimes E= \Hom _{\O_X} (\nu _*{\O_{\ti X}}, \O _X) \otimes
\Hom (\O_X, E) \to \Hom _{\O_X} (\nu _*{\O_{\ti X}}, E)=\nu
_*(\nu^{!}E)$$ given by composition of homomorphisms.
 Since $\nu ^*$ and $\nu _*$ are adjoint functors
this map induces
$$\nu ^*\C \otimes \nu ^* E \to \nu^{!}E.$$

Since $E$ is torsion free and $\C _{\ti X}= \nu^{\sharp} \C$ we
get
$$\C_{\ti X}\hot \nu^{\sharp} E \simeq \C_{\ti X}\cdot \nu^{\sharp} E\hookrightarrow \nu^{!}E.$$
Since this inclusion is an isomorphism at the generic point of $\ti X$ we have the following inequality
$$\mu (\C_{\ti X}\hot \nu^{\sharp} E) \le \mu (\nu^{!}E).$$
Now  Lemma \ref{normal-tensor-product} gives
$$\mu (\C_{\ti X}\hot \nu^{\sharp} E)= \mu (\nu^{\sharp} E)+\mu (\C_{\ti X}),$$
which implies the required inequality.
\end{proof}

\begin{Cor}
For any rank $r$ torsion free sheaf $E$ on $X$ we have
$$-\beta\le \mu (\nu ^{\sharp}E)-\mu (E) \le -\beta- \mu (\C_{\ti X}),$$ where the slopes are computed
with respect to $\O_X(1)$ on $X$ and $\nu^*\O_X(1)$ on $\ti X$.
\end{Cor}

\begin{proof}
The canonical map $E\to \nu_*(\nu^* E)$ leads to the inclusion
$$E\hookrightarrow \nu _*(\nu ^{\sharp}E).$$
This gives
$$\mu(E)\le \mu( \nu _*(\nu ^{\sharp}E))=\mu (\nu ^{\sharp}E)+\beta ,$$
where the last equality follows from proof of Lemma
\ref{useful}. This bounds the difference $\mu (\nu ^{\sharp}E)-\mu
(E)$ from below. To get the bound from the above it is sufficient
to use Lemma \ref{mu-mu2} and Corollary \ref{useful}.
\end{proof}

\medskip

\begin{Remark}
By Lemma \ref{mu-mu2} and the above corollary we have
$$ \mu (\nu^!{ E})\ge \mu (\nu^{\sharp} E)+ \mu (\C _{\ti X})\ge \mu (E)  -\beta+ \mu (\C_{\ti X}).$$
This allows  to take in Lemma \ref{mu-mu} $\alpha= -\mu (\C_{\ti
X})$. The proof of Lemma \ref{mu-mu} also gives a related and
explicit bound on $\alpha$.
\end{Remark}

\medskip

The above corollary can be used to prove  the following corollary:

\begin{Cor}
For any rank $r$ torsion free sheaf $E$ on $X$ we have
$$-\beta \le \mu _{\max}(\nu ^{\sharp}E)-\mu _{\max}(E)\le -\beta -\mu (\C_{\ti X}) .$$
\end{Cor}

\begin{proof}
If $G\subset E$ is a subsheaf of $E$ then $\nu^{\sharp}G\subset
\nu^{\sharp}E$ and hence
$$\mu (G)\le \mu (\nu^{\sharp}G)+
\beta\le \mu _{\max}(\nu^{\sharp}E)+\beta.$$
This proves that
$$\mu _{\max}(E)\le  \mu _{\max}(\nu^{\sharp}E)+
\beta.$$

Now if $F\subset \nu^{\sharp}E$ then by the proof of Lemma
\ref{mu-mu2} we have
$$\C_{\ti X}\hot F \subset \C_{\ti X}\hot \nu^{\sharp}
E\hookrightarrow \nu^{!}E.
$$
Together with Lemma \ref{normal-tensor-product} and Corollary \ref{useful2}, this gives
$$\mu (F)\le \mu _{\max}(\nu ^{!}E)-\mu (\C_{\ti X})\le
\mu _{\max}(E)-\beta-\mu (\C_{\ti X}),$$ which implies that
$$\mu _{\max}(\nu ^{\sharp}E)\le \mu _{\max}(E)-\beta -\mu (\C_{\ti X}).$$
\end{proof}

\medskip

Since $\nu^*(E_1\otimes E_2)=\nu^*E_1\otimes \nu^*E_2$ we have
$\nu^{\sharp}(E_1\hot E_2)=\nu^{\sharp}E_1\hot \nu^{\sharp}E_2$.
Therefore \cite[Introduction]{La3} or \cite[Lemma 3.2.1]{GLSS2}
imply the following proposition.

\begin{Prop}\label{mu_max-of-tensor-product}
There exists an explicit constant $\gamma$ (depending only on the
polarized  variety $(X, \O _X(1))$) such that for any two torsion free
sheaves $E_1$ and $E_2$ on $X$ of ranks $r_1, r_2$, respectively,
we have
$$\mu_{\max}(E_1\hot E_2)\le \mu_{\max}(E_1)
+\mu_{\max}(E_2)+(r_1+ r_2)\gamma.$$
\end{Prop}

\section{Honest singular principal bundles}

In this section $X$ is a $d$-dimensional projective variety
defined over an algebraically closed field $k$ with a fixed ample
line bundle $\O_X(1)$.

The main aim of this section is proof of the following
generalization of \cite[Proposition 3.4]{Sch4}:

\begin{Prop}\label{honest}
Assume that $X$ is Gorenstein (i.e., a Cohen--Macaulay scheme with
invertible dualizing sheaf $\omega _X$) and there exists a
$G$-invariant non-degenerate quadratic form $\varphi$ on $V$. Then
every degree $0$ singular principal bundle is an honest singular
principal bundle.
\end{Prop}

\begin{proof}
Let $(\A ,\tau)$ be a degree $0$ singular principal bundle.
 As in the proof of \cite[Proposition 3.4]{Sch4} one
can easily show that there exists an injective map $\A \to
\A^{\vee}$ induced by the form $\varphi$. By Lemma \ref{degree} we
see that the Hilbert polynomials of $\A$  and $\A^{\vee}$ are the
same up to the terms of order $O(m^{d-2})$. Hence $\A \to
\A^{\vee}$ is an isomorphism in codimension $1$. Now let us recall
that for each $x\in X$ two finitely generated modules over a local
ring $\O_{X,x}$ satisfying $S_2$ that coincide in codimension $1$
are equal. In particular, at each point $x$ where $\A$ is locally
free the map $\A \to \A^{\vee}$ is an isomorphism. As in the proof
of \cite[Proposition 3.4]{Sch4} this implies that
$$\sigma (U_{\A})
\subset {\mathbb I}{\rm som} (V\otimes {\O}_{U_{\A}},
\A^{\vee}\mid_{U_{\A}}) /G.$$
\end{proof}

\medskip

The following lemma generalizes a well known equality from smooth
varieties to singular ones.

\begin{Lemma}\label{Hilb-poly}
For any rank $r$ coherent sheaf $E$ and a line bundle $L$ we have
$$\deg (E\otimes L)= \deg E + r\, (L\cdot \O_X (1)^{d-1}).$$
\end{Lemma}

\begin{proof}
We use the notation from  Koll\'ar's book \cite[Chapter VI.2]{Ko}.
In particular, $K_i(X)$ stands for the subgroup of the
Grothendieck group of $X$ generated by subsheaves supported in
dimension at most $i$. We have
$$L\otimes E(m)=\sum _{i=0}^d c_1(L)^i\cdot E(m)$$
(see, e.g., \cite[Chapter VI.2, Lemma 2.12]{Ko}). On the other
hand, by \cite[Chapter VI.2, Corollary 2.3]{Ko} we have
$$E\equiv r\, \O_X \mod K_{d-1}(X).$$
Note that
$$L\otimes E(m)=E(m)+ r \, c_1(L)\cdot \O_X (m)
+c_1(L)\cdot (E-r\, \O_X)(m)+ \sum _{i\ge 2} c_1(L)^i\cdot E(m)$$
and $c_1(L)\cdot (E-r\, \O_X)+ \sum _{i\ge 2} c_1(L)^i\cdot E \in
K_{d-2}(X)$ by \cite[Chapter VI.2, Proposition 2.5]{Ko}. Therefore
by \cite[Chapter VI.2, Corollary 2.13]{Ko} we have
$$\chi (X,L\otimes E(m))=\chi (X, E(m))+r\chi
(X,c_1(L)\cdot \O_X(m))+O(m^{d-2}).$$ By the
 Riemann--Roch theorem
for singular varieties (see \cite[Corollary 18.3.1]{Fu}) we have
$$\begin{array}{ccl}
\chi (X,c_1(L)\cdot \O_X(m))&=&\chi (X,\O_X(m))- \chi
(X,L^{-1}(m))\\
& =& \int _X (\ch (\O_X(m))-\ch (L^{-1}(m)))\Td X\\
&=& (L\cdot \O_X (1)^{d-1}) \frac{m^{d-1}}{(d-1)!}+O(m^{d-2})
\end{array}$$
which, together with the previous equality, implies the lemma.
\end{proof}

\begin{Lemma}\label{degree}
If $X$ is Gorenstein  and $E$ is a torsion free sheaf on $X$ then
$$\deg E^{\vee}= - \deg E.$$
\end{Lemma}

\begin{proof}
Since $X$ is Cohen--Macaulay Serre's duality gives the equality
$$\chi (X, E)=(-1)^d \sum _{i=0}^d (-1)^i \dim \ext ^i(E, \omega_X).$$
The local to global Ext spectral sequence
$$  H^p (X,\Ext ^q (E,\omega _X))\Rightarrow \ext ^{p+q}(E, \omega _X)$$
implies that
$$\begin{array}{ccl}
\sum _{i=0}^d (-1)^i \dim \ext ^i(E, \omega_X) &=& \sum _{0\le p,q\le d} 
(-1)^{p+q}\dim H^p (X,\Ext ^q (E,\omega _X))\\
&=&
 \sum _{q=0}^d
(-1)^q \chi(X, \Ext _X^q(E, \omega_X)).
\end{array}
$$ Therefore we obtain
$$\chi(X, E(m))= (-1)^{d}  \sum _{q=0}^d (-1)^q
\chi(X, \Ext _X^q(E, \omega_X)\otimes \O_X(-m)).$$ By Lemma
\ref{Ext-support} we have $\dim \Ext _{X}^q(E,\omega_X)\le d-2$
for  $q> 0$, so by \cite[Chapter VI, Corollary 2.14]{Ko}
$$\chi(X, \Ext _X^q(E, \omega_X)\otimes \O_X(-m))=O (m^{d-2})$$
for $q>0$. Since $\omega_X$ is invertible $\Hom (E,\omega
_X)=E^{\vee}\otimes \omega_X$ and we get
$$\chi(X, E(m))= (-1)^{d} \chi(X, E^{\vee}\otimes \omega _X(-m))+O(m^{d-2}).$$
In particular, we have
$$\alpha_{d-1}(E^{\vee}\otimes \omega _X)=-\alpha_{d-1}(E).$$
Therefore by Lemma \ref{Hilb-poly}
$$
\begin{array}{ccl}
\deg E^{\vee}&=&\deg (E^{\vee}\otimes \omega _X)-r\, c_1(\omega
_X)\cdot c_1(\O_X (1))^{d-1}\\
&=&\alpha_{d-1}(E^{\vee}\otimes \omega _X)-r
\alpha_{d-1}(\O_X)-r\, c_1(\omega _X)\cdot c_1(\O_X (1))^{d-1}\\
&=& -\deg E-2r \alpha_{d-1}(\O_X)-r\, c_1(\omega _X)\cdot c_1(\O_X
(1))^{d-1}.
\end{array}
$$
Applying this equality for $E=\O_X$ we see that $$-2
\alpha_{d-1}(\O_X)-\, c_1(\omega _X)\cdot c_1(\O_X (1))^{d-1}=0,$$
so $\deg E^{\vee}= - \deg E$.
\end{proof}

\section{Semistable reduction for singular
principal $G$-bundles}

The following global boundedness of swamps on singular varieties
can be proven in the same way as in the case of smooth varieties
(see \cite[Theorem 4.2.1]{GLSS1}, \cite[Theorem 3.2.2]{GLSS2} or
\cite[Theorem 2.3.4.3]{Sch5}). The only difference is that we need
Proposition \ref{mu_max-of-tensor-product} (instead of, e.g.,
\cite[Lemma 3.2.1]{GLSS2}).

\begin{Thm}
Let us fix a polynomial $P$, integers $a$, $b$ and a class $l$ in
the N\'eron--Severi group of $X$. Then the set of isomorphism
classes of torsion free sheaves $\A$ on $X$ with Hilbert
polynomial $P$ and such that there exists a positive rational
number $\overline{\delta}$ and a slope
$\overline{\delta}$-semistable $\rho_{a,b}$-swamp $(\A, L,
\varphi)$ with $L$ of class $l$ is bounded.
\end{Thm}

This boundedness result implies the following semistable reduction
theorem (see \cite[Theorem 5.4.4]{GLSS1}, \cite[Theorem
4.4.1]{GLSS2} or \cite[Theorem 2.4.4.1]{Sch5}). We skip the proof
as it is the same as in the smooth case.

\begin{Thm}
Assume that $k$ has characteristic zero. Then there exists a
polynomial $\delta _{\infty}$ such that for every positive
polynomial $\delta>\delta_{\infty}$ every $\delta$-semistable
pseudo $G$-bundle $(\A, \tau)$ is a singular principal $G$-bundle.
\end{Thm}

Let us recall that a singular principal $G$-bundle is semistable
if and only if the associated pseudo $G$-bundle is
$\delta$-semistable for $\delta>\delta_{\infty}$ (see
\cite[Theorem 5.4.1]{GLSS1}). Therefore the above semistable
reduction theorem and Theorem \ref{Bhosle-Schmitt} imply the
following corollary.

\begin{Cor}
Assume that $k$ has characteristic zero and let us fix a
polynomial $P$. Then there exists a projective moduli space
$M^{\rho}_{X,P}$ for
semistable principal $G$-bundles $(\A , \tau)$ on $X$ such that
$\A$ has Hilbert polynomial $P$.
\end{Cor}

\medskip

Now let us consider the relative case. Let $f:X\to S$ be a flat,
projective morphism of $k$-schemes of finite type with integral
geometric fibers. Assume that $k$ has characteristic zero and fix
a polynomial $P$.

\begin{Thm}
Let us fix a faithful representation $\rho\colon G \to {\rm
GL}(V)$ of the reductive algebraic group $G$.

\begin{enumerate}
\item There exists a projective moduli space $M^{\rho}_{X/S, P}\to S$ for $S$-flat families
of semistable singular principal $G$-bundles on $X\to S$ such that
for all $s\in S$ the restriction $\A |_{X_s}$ has Hilbert
polynomial $P$.
\item Let $P$ correspond to sheaves of degree $0$.
If the fibres of $f$ are Gorenstein and there exists a
$G$-invariant non-degenerate quadratic form $\varphi$ on $V$ then
$M^{\rho}_{X/S, P}\to S$ parameterizes only honest singular principal $G$-bundles.
\end{enumerate}
\end{Thm}

The first part of this theorem follows directly from the above
corollary (rewritten in the relative setting). The second part is
a direct consequence of Proposition \ref{honest}.
Since proof in the relative setting is essentially the same as usual
(cf. \cite[Theorem 4.3.7]{HL}) we skip the details.

\medskip

\section*{Acknowledgements}

The author would like to thank Alexander Schmitt for useful
conversations. He would also like to thank  the Alexander von
Humboldt Foundation for supporting, via the Bessel Research Award,
his visit to the University of Duisburg-Essen, where most
of this paper was written.
The author was partially supported by a Polish MNiSW grant
(contract number N N201 420639).

\end{document}